\documentclass{amsart}
\usepackage{mathtools,amsfonts,amsmath,amsthm,amssymb,xcolor}
\usepackage{enumitem}
\usepackage{fullpage}
\usepackage{hyperref}
\usepackage{cleveref}
\usepackage{xcolor}
\newtheorem{thm}{Theorem}[section]

\newtheorem{lem}[thm]{Lemma}
\newtheorem{cor}[thm]{Corollary}

\newtheorem{thmintro}{Theorem}

\theoremstyle{definition}

\theoremstyle{remark}

\newcommand{\B}[1]{\overline{#1}}
\newcommand{\C}{\mathbb{C}}
\newcommand{\R}{\mathbb{R}}
\newcommand{\rb}{\rangle}
\newcommand{\lb}{\langle}

\newcommand{\bgm}[2]{\lb#1,#2\rb}

\newcommand{\U}{\mathsf{U}}

\AtBeginDocument{%
   \def\MR#1{}
}

\title{Hermitian manifolds with flat Gauduchon connections}

\author{Ramiro A.~ Lafuente}
\address{School of Mathematics and Physics, The University of Queensland, St Lucia QLD 4072, Australia}
\email{r.lafuente@uq.edu.au}

\author{James Stanfield}
\address{School of Mathematics and Physics, The University of Queensland, St Lucia QLD 4072, Australia}
\email{james.stanfield@uq.net.au}

\begin{document}

	\begin{abstract}
	We complete the classification of compact Hermitian manifolds admitting a flat Gauduchon connection. In particular, we establish a conjecture of Yang and Zheng, showing that apart from the cases of a flat Chern or Bismut connection, such manifolds are K\"ahler. More generally, we prove the same result holds when the flatness assumption is replaced by the so-called K\"ahler-like condition, proving a conjecture of Angella, Otal, Ugarte and Villacampa. We also treat the non-compact case.

	\end{abstract}

	\maketitle

	\section{introduction}
	Let $(M,J,g)$ be a Hermitian manifold. Unless $g$ is K\"ahler, the complex structure $J$ is not parallel with respect to the Levi--Civita connection. It is therefore more natural to consider \emph{Hermitian connections}: those for which $g$ and $J$ are parallel.
	In general,  Hermitian connections form an infinite dimensional affine space. However, by imposing natural representation-theoretic constraints on the torsion, Gauduchon  identified  in \cite{GauduchonHermitianConnections} a distinguished line of ``canonical'' Hermitian connections, which we will call \emph{Gauduchon connections}, given by
	\begin{equation}
		\label{eqn:sNabla}
		\left\{\nabla^s :=  (1- \tfrac{s}{2}) \, \nabla^c + \tfrac{s}{2} \, \nabla^b \,\, : \,\, s \in \R\right\}.
	\end{equation}
	Here, $\nabla^c = \nabla^0$ and $\nabla^b=\nabla^2$ are respectively the Chern and Strominger-Bismut \cite{Bis89} connections of $(M,J,g)$. Recall that the Chern connection is the unique Hermitian connection for which the $(1,1)$-part of the torsion vanishes, and the Bismut connection is the unique Hermitian connection with totally skew-symmetric torsion. 
	
	When $(M,J,g)$ is K\"ahler, all Gauduchon connections coincide with the Levi--Civita connection $\nabla^{\rm LC}$. On the other hand, they are pairwise distinct on non-K\"ahler Hermitian manifolds. Remarkably, Gauduchon's line  also includes other distinguished connections, such as the \emph{Lichnerowicz connection} \cite{Lic57} $\nabla^1 = \frac{1}{2}\left(\nabla^{\operatorname{LC}} - J\nabla^{\operatorname{LC}}J\right)$, the Hermitian part of the Levi--Civita connection (sometimes also called the \emph{first canonical}, or \emph{associated} connection); Liebermann's \emph{conformal connection} \cite{Lib54} $\nabla^{\frac12}$ whose torsion satisfies the Bianchi identity; and the \emph{minimal connection} $\nabla^{\frac{2}{3}}$ \cite{GauduchonHermitianConnections},  whose torsion has minimal norm among all Hermitian connections.


	A  fundamental question in non-K\"ahler Hermitian geometry due to  S.T.~Yau \cite[Problem 87]{Yau93}, asks  whether one can say something nontrivial about a compact Hermitian manifold whose holonomy is a proper  subgroup of $\U(n)$. Yau remarks that the problem is that the connection need not be the Levi-Civita connection. It seems natural to consider this problem for Gauduchon connections. In this direction, our main result answers Yau's question in the case of a discrete holonomy group:
	\begin{thmintro}\label{thm_main}
		Any compact Hermitian manifold with a flat Gauduchon connection  $\nabla^s$  is K\"ahler, unless $s= 0$ or $2$.
	\end{thmintro}

	This also yields an affirmative answer to a conjecture by Yang  and Zheng (\cite[~Conjecture 1.1]{YangZhengGuaduchonFlat}).  Moreover, our proof only uses the so-called \emph{K\"ahler-like} condition (see \eqref{eqn_Bianchi0} and \eqref{eqn_type} below) and thus establishes a much more general conjecture proposed by Angella, Otal, Ugarte and Villacampa (\cite[Conjecture 2]{AOUV}), see Theorem \ref{thm_main_body} for the precise statement.   Recall that the cases $s=0, 2$ are well-understood. By a classical result of Boothby \cite{Boothby58} a compact Chern-flat Hermitian manifold is covered by a simply-connected complex Lie group with left-invariant Hermitian metric. On the other  hand, Wang, Yang and Zheng showed in \cite{WangYangZhengBismutFlat} that the universal cover of a  Bismut-flat Hermitian manifold is an open complex submanifold of a \emph{Samelson space}, that is, a Lie group with left-invariant complex structure and bi-invariant metric.

	Theorem \ref{thm_main} was previously known when $n=2$ (\cite[~Theorem 1.5]{YangZhengGuaduchonFlat}), and for arbitrary dimension under the additional assumption that $s \notin (4-2\sqrt{3},1)\cup(1,4+2\sqrt{3})$ (\cite[~Theorem 1.6]{YangZhengGuaduchonFlat} and \cite[~Theorem 1]{ZhaoZhengKahlerLike}; see also \cite[Corollary 1.5]{FZ19}). It was also essentially known when $g$ is  locally conformally K\"ahler \cite[~Theorem 1.7]{YangZhengGuaduchonFlat}, or \emph{balanced} \cite[~Proposition 1.8]{YangZhengGuaduchonFlat}.  Hermitian manifolds with flat Lichnerowicz connection were previously  studied in \cite{ABD12,GK06}.

	It is worthwhile mentioning that the flatness assumption cannot be relaxed to the vanishing of the (first) Ricci form: for each $s\in \R$ there exist non-K\"ahler Hermitian manifolds with Ricci-flat Gauduchon connection $\nabla^s$. Indeed, by \cite{Vez13}, on a balanced Hermitian nilmanifold the Ricci form of $\nabla^s$ vanishes for all $s$. Some explicit examples in complex dimension $3$ are provided in \cite{FPS04}. Note that in general, different contractions of the curvature of $\nabla^s$ give rise to different Ricci-type tensors. It would be interesting to find non-K\"ahler examples of Hermitian manifolds for which these other Ricci curvatures vanish.

	Regarding  compactness, recall that Boothby's result \cite{Boothby58} for Chern-flat manifolds does not extend to the non-compact case. Thus, it is natural to ask whether there exist non-compact, non-K\"ahler Hermitian manifolds with flat Gauduchon connections $\nabla^s$, $s\notin\{0,2\}$. In this direction, we first remark that our proof of \Cref{thm_main} only uses compactness to guarantee the existence of a maximum for a certain natural real-valued function on $M$ (see the discussion of the proof below). Thus, we immediately deduce an analogous statement to \Cref{thm_main} for non-compact homogeneous Hermitian manifolds, and more generally for manifolds where the group of biholomorphic isometries has a compact fundamental domain: see \Cref{cor_cocompact_symmetries}. This generalises the main result in \cite{VezzoniYangZhengLieGroups}, where the authors assume homogeneity and the fact that $\nabla^s$ has parallel torsion.

	Even more generally, without any symmetry assumptions, our main result in the non-compact case is: 

	\begin{thmintro}
		\label{thm_B}
		Any non-compact Hermitian manifold with a flat Gauduchon connection  $\nabla^s$ with $s\notin\{ 0, \tfrac23 , \tfrac45, 2 \}$  is K\"ahler. 
	\end{thmintro}

	As with \Cref{thm_main}, \Cref{thm_B} holds when flatness is replaced by the K\"ahler-like assumption (again, see \Cref{thm_main_body}). Regarding the remaining possible exceptions for $s$, we remind the reader that $\nabla^{2/3}$ is the minimal connection. On the other hand, we are not aware of any geometric significance for $\nabla^{4/5}$. It would be interesting to find non-K\"ahler examples of Hermitian manifolds on which either of these connections is flat (or K\"ahler-like). Moreover, Hermitian connections beyond the Gauduchon line could also be of interest. In particular, it would be natural to ask if a given non-K\"ahler manifold can admit a flat Hermitian connection (other than the Chern or Bismut connections).

	We now briefly summarise the main ideas involved in the proof of Theorems \ref{thm_main} and \ref{thm_B}. Let $T$ be the torsion tensor of the Chern connection and $\eta \in \Omega^{1,0}(M)$ the \emph{torsion $(1,0)$-form} given (up to scaling) by the $(1,0)$ part of the trace of $T$. As observed in \cite{YangZhengGuaduchonFlat,ZhaoZhengKahlerLike}, flatness (or more generally K\"ahler-likeness) of $\nabla^s$ together with the general Bianchi identity imply that $\nabla^sT = Q_s(T)$, where $Q_s(T)$ is a quadratic expression in the torsion depending on $s$. Denoting by $R^s = [\nabla^s_{\cdot},\nabla^s_{\cdot}] - \nabla^s_{[\cdot,\cdot]}$ the curvature of $\nabla^s$, it follows that $R^s(\cdot,\cdot)T = C_s(T)$ is a tensor which is cubic in the torsion. We then consider four $\C$-valued functions on $M$: $\bgm{\partial |\eta|^2}{\B \eta}$, $A,B,C \colon M \to \C$, where $A,B$ and $C$ are contractions of $T\otimes T\otimes T \otimes T$ (see: \Cref{eqn:ABC}). The condition $R^s = 0$ (or $R^s$ K\"ahler-like) applied to contractions of the tensor $R^s(\cdot,\cdot)T\otimes \B{\eta}$ yields four identities, linear in $A,B,C$ and $\bgm{\partial|\eta|^2}{\B \eta}$. The linear system formed is non-singular for $s \notin \{0,1/2,2/3,4/5,1,2\}$, and for such cases we find in particular that $C = 0$ which is enough to conclude that $g$ is K\"ahler. The cases $s \in \{1/2,1\}$ are dealt with in a straightforward manner and in fact were treated in \cite{ZhaoZhengKahlerLike}. This proves \Cref{thm_B}. \Cref{thm_main} follows by considering a point $p \in M$ where the function $|\eta|^2$ attains a maximum. At such a point, $\bgm{\partial |\eta|^2}{\B \eta} = 0$ and so the previously formed linear system is overdetermined, which forces $C = 0$ at $p$. It follows that $\eta \equiv 0$, which is again enough to conclude that $T \equiv 0$ and hence $g$ is K\"ahler.

	\subsubsection*{Acknowledgements}
	The first-named author was supported by an Australian Research Council DECRA fellowship (project ID DE190\-101063). The second named author was supported by an Australian Government Research Training Program (RTP) Scholarship.

\section{Preliminaries}
	\label{sec:Prelim}
	
	In this section we fix our notation and recall some important results.  Let $(M^{2n},J,g)$ be a Hermitian manifold and for each $s\in \R$, let $\nabla^s$ denote the $s$-Gauduchon canonical connection defined in \eqref{eqn:sNabla}. Denote by $\omega := g(J\cdot,\cdot)$ the fundamental form associated to the Hermitian structure.  Suppose $U\subset M$ is open and $\{e_i\}_{i=1}^n\subset \Gamma(U,T^{1,0}M)$ is a local unitary (1,0)-frame, and let $\{e^i\}_{i=1}^n \subset \Omega^{1,0}(U)$ be the associated dual frame. In the rest of the article, we adopt Einstein's summation convention over repeated indices. A comma followed by an index will always denote covariant differentiation with respect to $\nabla^s$ in the direction of the corresponding basis vector.

	Let $T^s(X,Y) := \nabla^s_{X}{Y} - \nabla^s_{Y}{X} - [X, Y]$ be the torsion of $\nabla^s$. We reserve the notation $T:= T^0$ for the torsion of the Chern connection. In our fixed local unitary frame, $T$ is given by 
	\[
			T(e_i, e_j) = 2T_{ij}^k e_k, \qquad T(\B{e_i}, \B{e_j}) = 2\B{T_{ij}^k} \B{e_k}.
	\]
	where the functions $T_{ij}^k : U \to \C$ are defined via
	\[
	T_{ij}^k := \frac{1}{2}\bgm{T(e_i,e_j)}{\B{e_k}}\colon U \to \mathbb{\C};\qquad 1\leq i,j,k\leq n.
	\]
	(The $1/2$ factor reflects our wedge product convention: $\alpha \wedge \beta = \alpha \otimes \beta - \beta \otimes \alpha$.)
	The \emph{Lee form} of $(M,J,g)$ is the 1-form  $\theta \in \Omega^1(M)$ defined via 
	\[
		d \omega^{n-1} = \theta \wedge \omega^{n-1}.
	\]
	Inspired by \cite{Mich82}, but following the sign conventions in \cite{YangZhengGuaduchonFlat}, we define the torsion $(1,0)$-form $\eta := -\tfrac12 \theta^{1,0}$, which in  terms of the unitary frame  is given by
	\[
	\eta =\eta_je^j = T_{kj}^ke^j \in \Omega^{1,0}. 
	\]
	
	By viewing the torsion as a $TM$-valued $2$-form,  $J$ induces on it a type decomposition as usual: 
	\[
		T^s = (T^s)^{2,0} + (T^s)^{1,1} + (T^s)^{0,2}.
	\]
	Since the $(0,2)$-part of the torsion of any Hermitian connection is the Nijenhuis tensor \cite{GauduchonHermitianConnections}, we have $(T^s)^{0,2} = 0$ for all $s\in \R$.  Moreover, from the well-known facts $(T^0)^{1,1} = (T^1)^{2,0} = 0$, $T^2 \in \Omega^3 M$ \cite{GauduchonHermitianConnections}, we may compute $T^s$ in terms of $T = T^0$:

	\begin{lem}
	\label{lem:torsionComponents}
		The torsion $T^s$ of $\nabla^s$ may be recovered from $T$ by
		\[
			T^s(e_i,e_j)= 2(1-s)T_{ij}^ke_k,\qquad T^s(\B{e_i},e_j) = s \, \B{T^j_{ik}}e_k - s \, T^i_{jk}\B{e_k}, \qquad 1\leq i,j \leq n.
		\]
	\end{lem}
	\begin{proof}
		By definition, $\nabla^s = (1-s/2) \, \nabla^0 + (s/2) \,  \nabla^2 =  (1-s) \, \nabla^0 + s\nabla^1$. From $(T^1)^{2,0} = 0$ we deduce $(T^s)^{2,0} = (1-s)T$, which proves the first claim. 

		For the second part, we use $T^{1,1} = 0$  and the fact that $T^2$ is totally skew-symmetric to get
		\[
		\bgm{T^s(\B{e_i},e_j)}{\B{e_k}} = \frac{s}{2}\bgm{T^2(\B{e_i},e_j)}{\B{e_k}} = -\frac{s}{2}\bgm{T^2(\B{e_i},\B{e_k})}{e_j} = \frac{s}{2}\B{\bgm{T(e_i,e_k)}{\B{e_j}}} = s \, \B{T^j_{ik}},
		\]
		as claimed.
	\end{proof}

	Following \cite{AOUV},  we call a Gauduchon connection $\nabla^s$ \emph{K\"ahler-like} if its curvature tensor $R^s$ has the same symmetries as in the K\"ahler case. Namely, it satisfies the (torsion-less) first Bianchi identity
	\begin{equation}\label{eqn_Bianchi0}
		R^s(X,Y)Z + R^s(Y,Z)X + R^s(Z, X) Y = 0,  \qquad \hbox{for all } X, Y, Z \in TM,
	\end{equation}
	and the so-called type condition 
	\begin{equation}\label{eqn_type}
		R^s(JX,JY) = R^s(X,Y),\qquad \hbox{for all } X,Y \in TM.
	\end{equation}
	Equivalently, $\nabla^s$ is K\"ahler-like if $R^s(Z,W, \cdot ,\cdot ) = R^s(\cdot, \cdot, Z, W) = 0$ and $R^s(Z,\cdot,W,\cdot) = R^s(W,\cdot,Z,\cdot)$ for all $Z,W\in T^{1,0} M$ (here $R^s(\cdot,\cdot,\cdot,\cdot) = g(R^s(\cdot,\cdot) \cdot, \cdot)$). Notice that a flat connection is trivially K\"ahler-like.

	The starting point for our computations is the following set of formulas for the covariant derivatives of $T$, which may be found in \cite[Lemmas 2, 3 $\&$ 5]{ZhaoZhengKahlerLike} (notice that, in that article, the parameter $r$ for the Gauduchon line is related to ours via $r = 1-s$) and \cite[Lemma 3.1]{VezzoniYangZhengLieGroups}. Since our assumptions here are sligthly weaker, for convenience of the reader we indicate how the proof goes, referring to the afforementioned articles for details.

	\begin{lem}[\cite{ZhaoZhengKahlerLike,VezzoniYangZhengLieGroups}]
		\label{lem:torsionIdentities}
		Let $(M^{2n},J,g)$ be a Hermitian manifold with Gauduchon connection $\nabla^s$, $s\neq 0$, satisfying the torsion-less Bianchi identity \eqref{eqn_Bianchi0}. Then, in any local unitary frame $\{e_i\}_{i=1}^n$ we have:
		\begin{enumerate}
			\item  $T^k_{ij,l} = -(s-2) \, T^r_{ij}T^k_{rl}$ \label{lem:torsionIdentitiesNoBar} ,
			\item $0= (s-1)\left(T_{j k}^{r}T_{i r}^{l}+ T_{k i}^{r}T_{j r}^{l}+T_{i j}^{r}T_{k r}^{l}\right)$ \label{lem:torsionIdentitiesCyclic},
			\item $\begin{aligned} 
				c \, T_{i j, \bar{l}}^{k} = a \, T_{i j}^{r} \B{T_{k l}^{r}}+b\left(T_{i r}^{k} \B{T_{l r}^{j}}-T_{j r}^{k} \B{T_{l r}^{i}}\right)+ s^{3}\left(T_{i r}^{l} \B{T_{k r}^{j}}-T_{j r}^{l} \B{T_{k r}^{i}}\right) ,
				\end{aligned}$  \label{lem:torsion}
		\end{enumerate}
		for all  $1\leq i,j,k,l \leq n$. Here a comma denotes covariant differentiation with respect to $\nabla^s$, and
		\[
		a = a(s) := -4s(s-1)^2,\quad b = b(s) :=-s(5s^2-10s+4), \quad c = c(s) := 4(s-1)(2s-1).
		\]
	\end{lem}

	\begin{proof}
 	Any  connection satisfies a Bianchi identity, which in general involves the torsion (see e.g. \cite[Ch.~III, Thm.~5.3]{KN}). Together with the torsion-less Bianchi assumption \eqref{eqn_Bianchi0} this yields
 	\[
 		\sum_{\sf{cyc}} \left( T^s(T^s(X,Y), Z) + \left( \nabla^s_X T^s \right)(Y,Z) \right) = \sum_{\sf{cyc}} R^s(X,Y) Z = 0,
 	\]
 	where the sum is over a cyclic permutation of $(X,Y,Z)$. 
 	 This identity extends naturally to the complexified tangent space. We first evaluate it on three $(1,0)$ frame vectors $e_i, e_j, e_k$ and use \Cref{lem:torsionComponents} to get
 	 \begin{equation}\label{eqn_30part}
 	 	\sum_{\sf{cyc}(i,j,l)} \left( T_{ij,l}^k +2 (1-s) \, T_{ij}^r T_{rl}^k \right) = 0.
 	 \end{equation}
 	 On the other hand, evaluating on ${e_i}, {e_j}, \B{e_k}$ and equating the $(1,0)$-parts gives
 	 \begin{equation}\label{eqn_10part}
 	 	-s \left( T_{jl,i}^k + T_{li,j}^k \right) = (2s - s^2) T_{ij}^r T_{rl}^k - s^2 \sum_{\sf{cyc}(i,j,l)} T_{ij}^r T_{rl}^k,
 	 \end{equation}
 	 while the $(0,1)$ part is
 	 \begin{equation}\label{eqn_01part}
 	 		2(s-1) T_{ij,\B{k}}^l + s \, \left(\B{T_{kl,\B{i}}^j} - \B{T_{kl,\B{j}}^i }\right) =  -2s(1-s) \left( T_{ij}^r \B{T_{kl}^r} + T_{ri}^l  \B{T_{kr}^j} - T_{rj}^l \B{T_{kr}^i}\right)+ s^2 \left( T_{jr}^k \B{T_{rl}^i} - T_{ir}^k \B{T_{rl}^j}\right). 
 	 \end{equation}
 	 Using \eqref{eqn_10part} for cyclic permutation of the indices $(i,j,l)$ and adding all up we obtain
 	 \[
 	 	2s \sum_{\sf{cyc}(i,j,l)} T_{ij,l}^k = 2s(2s-1) \sum_{\sf{cyc}(i,j,l)} T_{ij}^r T_{rl}^k.
 	 \]
 	 Combining with \eqref{eqn_30part} this immediately implies Items (1) and (2). Item (3) follows from \eqref{eqn_01part} by using the algebraic trick explained in the proof of \cite[Lemma 3.1]{VezzoniYangZhengLieGroups}.
	\end{proof}

	Taking traces in \Cref{lem:torsionIdentities} yields the following useful identities for $\eta$:
	\begin{cor} Under the assumptions of \Cref{lem:torsionIdentities}, for all $1\leq i,k,j,l \leq n$ we have that:
		\label{cor:etaIdentities}
		\begin{enumerate}
			\item $\eta_{k,l} = -(s-2)T^r_{ik}T^i_{rl}$,
			\item $0 = (s-1) \, T^i_{jk}\eta_i $,
			\item $c\, \eta_{j,\B l} = (a - b)\, T^p_{kj} \, \B{T^p_{kl}} + b  \, \eta_p  \, \B{T^j_{lp}} + s^3 \, \left(T^l_{kp}\, \B{T^j_{kp}} - T^l_{jp} \, \B {\eta_p}\right)$,
			\item  \label{cor:etaIdentities:normT^2}
			 $2(2s-1)\, \eta_{r,\B r} = s^2\, |T|^2 + s(2-3s) \, |\eta|^2$,
			\item $c \, \eta_{j,\B l} \, \B {\eta_{j}} = (a - b) \, T^p_{kj} \, \B{T^{p}_{kl}} \, \B{\eta_j}$.
		\end{enumerate}
		In particular, if in addition $(M^{2n}, J, g)$ is  balanced (i.e.~$\eta \equiv 0$) then it must be K\"ahler.
	\end{cor}
	\begin{proof}
	The first three items come from taking a trace of items (1)-(3) in \Cref{lem:torsionIdentities}. Item (4) is just the trace of (3). Item (5) comes from multiplying (3) by $\B{\eta_j}$ and applying  (1). Finally, the last assertion follows from \Cref{cor:etaIdentities:normT^2} and the fact that $s\neq 0$.
	\end{proof}
	
	In order to obtain more algebraic identities using the K\"ahler-like condition on $R^s$, we will make use of the $\C$-linear extension of the \emph{second-covariant derivative operator}. For a complex tensor bundle $E \to M$ associated to $TM^\C$, this is defined by
	\[
	(\nabla^s)^2 := \nabla^s \circ \nabla^s \colon E \to T^*M\otimes E \to T^*M \otimes T^*M \otimes E.
	\]
	For a $(1,0)$-vector field $X = X^ie_i\in \Gamma(U,T^{1,0}U)$, the frame expression for $(\nabla^s)^2 X$ is given by
	\[
	(\nabla^s)^2X = X^i_{,lk}e^k\otimes e^l\otimes e_i + X^i_{,\B l k}e^k \otimes \B{e^l}\otimes e_i + X^i_{,l \B{k}}\B{e^k} \otimes e^l\otimes e_i.
	\]
	Similar frame expressions for other tensors may be obtained analogously. This operator is related to the curvature and torsion of $\nabla^s$ by the following formula:
	\begin{lem}For all $X,Y \in TM$, $(\nabla^s)^2_{X,Y} - (\nabla^s)^2_{Y,X} = R^s(X,Y) - \nabla^s_{T^s(X,Y)}$.
	\end{lem}
	\begin{proof}
		By the Leibniz rule for covariant derivatives, we have that
		\[
		(\nabla^s)^2_{X,Y} = \nabla^s_X \left(\nabla(\cdot)\right)(Y) = \nabla^s_X(\nabla^s_Y (\cdot)) - \nabla^s_{\nabla^s_XY}.
		\]
		The result now follows from the definitions of $R^s$ and $T^s$.
	\end{proof}
	By \Cref{lem:torsionComponents}, this yields:
	\begin{cor}\label{cor:secondDerivComm}
		For all indices $1\leq i,j \leq n$,
		\begin{equation*}
			R^s(e_i,e_j) = (\nabla^s)^2_{e_i,e_j} - (\nabla^s)^2_{e_j,e_i} + 2(1-s) \, T_{ij}^k\nabla^s_{e_k},
		\end{equation*}
		and,
		\begin{equation*}
			R^s(e_i,\B{e_j}) = (\nabla^s)^2_{e_i,\B{e_j}} - (\nabla^s)^2_{\B{e_j},e_i} - s \, \B{T^{i}_{jk}}\nabla^s_{e_k} + s \, T^j_{ik}\nabla^s_{\B{e_k}}.
		\end{equation*}
	\end{cor}

\section{Gauduchon connections with $s\neq 2/3, 4/5$}

	Our aim in this section is to prove 
	
	\begin{thm}\label{thm_main_body}
		Let $(M,J,g)$ be a Hermitian manifold. Assume that $\nabla^s$ is K\"ahler-like for some $s \neq 0,2$, and that either $|\eta|^2$ attains a maximum, or $s\neq \tfrac23, \tfrac45$. Then, $(M,J,g)$ is K\"ahler.
	\end{thm}
	\Cref{thm_main} immediately follows from \Cref{thm_main_body} by compactness. More generally, we yield the same rigidity when the manifold admits a cocompact group of symmetries.
	\begin{cor}\label{cor_cocompact_symmetries}
		Let $(M,J,g)$ be a Hermitian manifold with $\nabla^s$ K\"ahler-like for some $s\neq 0,2$. If the group of biholomorphic isometries has compact fundamental domain, then $(M,J,g)$ is K\"ahler.
	\end{cor}

	Our proof of \Cref{thm_main_body} will require computing second covariant derivatives of the torsion in order to obtain new algebraic identities. Let us introduce some convenient notation in order to aid in these computations. We define first the following contractions of  $T\otimes\B{T} \in T^{1,0}M \otimes T^{0,1}M\cong T^{1,0}M \otimes \Lambda^{1,0}M$ as:
	\[
		U := \frac{1}{2}T(\cdot,\B{\eta^\sharp}),\qquad V := \frac{1}{4}\bgm{T(e_k,\cdot)}{\B{T(e_k,\cdot)}},\qquad W := \frac{1}{4}T(e_k,e_l)\otimes \B{T(e_k,e_l)}.
	\]
	 Recall that for a co-vector $\alpha \in T^*M^\C$, $\alpha^\sharp \in TM^\C$ is the vector satisfying $\alpha = \bgm{\cdot}{\alpha^\sharp}$. The components of $U,V$ and $W$ in a local unitary frame are respectively given by
	\[
		U^i_j = T^i_{jk}\B{\eta_k},\qquad V_{i\B{j}} = T^r_{ik}\B{T^r_{jk}},\qquad W^{i\B j} := T^i_{kl}\B{T^j_{kl}}.
	\]
	Notice that $V_{i\B j} = \B{V_{j\B i}}$ and $W^{i\B j} = \B{W^{j\B i}}$. Moreover, by \Cref{cor:etaIdentities},
	\begin{equation}
		\label{eqn:etaBarDeriv}
		c\, \eta_{j,\B l} = (a - b)V_{j\B l} + b  \,  \B{U^j_l} + s^3 \, (W^{l\B j} - U^l_j),
	\end{equation}
	for all $1 \leq j,l\leq n$. Next, we define four complex-valued functions on $(M,J,g)$,  contractions of $T \otimes T \otimes \B{T} \otimes \B{T}$, that will play a central role in our computations:
	\[
	A := V_{k\B p}U^k_p,\qquad \tilde A := W^{k\B r}U^r_k, 
	\qquad B := U^k_pU^p_k = \operatorname{tr}(U^2),
	\qquad C := V_{i\B j}\eta_j\B{\eta_i} = U^k_p\B{U^k_p} = |U|^2.
	\]
	Equivalently, using the definition of $U,V$ and $W$, we have:
	\begin{equation}\label{eqn:ABC}
		A = T^k_{pj}T^r_{kl}\B{T^r_{pl}}\B{\eta_j}, 
		\qquad \tilde A = T^k_{pl}T^r_{kj}\B{T^r_{pl}}\B{\eta_j}, 
		\qquad B = T^p_{kj}T^k_{pr}\B{\eta_j}\B{\eta_r},
		\qquad C = T^p_{kj}\B{\eta_j}\B{T^p_{kr}}\eta_r.
	\end{equation}

	A direct application of \Cref{lem:torsionIdentitiesCyclic} of \Cref{lem:torsionIdentities} yields the following algebraic and first-order relations between the above contractions:

	\begin{lem}\label{lem:normOfTDeriv}
		If $\nabla^s$ satisfies the torsion-less Bianchi identity \eqref{eqn_Bianchi0} and $s \notin \{0,1\}$, then the following hold:
		\begin{enumerate}
			\item \label{eqn_twoA} $\tilde A = 2A$;
			\item \label{eqn_delT^2} $c\, \bgm{\partial |T|^2}{\B{\eta}} 
			  =2(a-b - c(s-2))A =
			 -2(s-2)(7s^2-12s+4)A$;
			\item  \label{eqn_deleta^2} $c\, \bgm{\partial |\eta|^2}{\B \eta} = -c(s-2) \, B + (a-b) \, C$.
		\end{enumerate}
	\end{lem}
	\begin{proof}
		\Cref{eqn_twoA} follows immediately from \Cref{lem:torsionIdentities}, \Cref{lem:torsionIdentitiesCyclic} (using that $s\neq 1$).

		Regarding \Cref{eqn_delT^2}, we compute directly in a local unitary frame $\{e_i\}_{i=1}^n$, using \Cref{lem:torsionIdentities},  \Cref{cor:etaIdentities} and \Cref{eqn_twoA} above: 
		\[
		\begin{split}
			c\, \bgm{\partial|T|^2}{\B \eta} &= c \,  (T^k_{ij}\B{T^k_{ij}})_{,l} \, \B{\eta_l}\\
			&= -c(s-2)T^r_{ij}T^k_{rl}\B{T^k_{ij}}\B{\eta_l}\\
			& \quad \, + T^k_{ij}\left(a\B{T_{i j}^{r}} T_{k l}^{r}+b\left(\B{T_{i r}^{k}} T_{l r}^{j}-\B{T_{j r}^{k}} T_{l r}^{i}\right)+ s^{3}\left(\B{T_{i r}^{l}} T_{k r}^{j}-\B{T_{j r}^{l}} T_{k r}^{i}\right)\right)\B{\eta_ l}\\
			&= (-c(s-2) + a)\tilde A - bA - bA \\
			& = - 2 c (s-2) A  + 2 (a-b) A\\
			&= -2(s-2)(7s^2-12s+4)A.
		\end{split}
		\]

		Finally, for \Cref{eqn_deleta^2} we use again \Cref{cor:etaIdentities} repeatedly:
		\[
		\begin{split}
		c\bgm{\partial|\eta|^2}{\B \eta} &= c \, (\eta_{i}\B{\eta_i})_{j}\B{\eta_j} = c \, \eta_{i,j}\B{\eta_i\eta_j} + c \, \B{\eta_{i,\B j}\B{\eta_i}} \,  \B{\eta_j} \\
		&= -c(s-2)T_{ki}^rT^k_{rj}\B{\eta_i\eta_j} + (a-b) \, \B{T^p_{ki}\B{T^p_{kj}}\B{\eta_i}} \,  \B{\eta_j} \, = \, -c(s-2)B + (a-b)C.
		\end{split}
		\]
	\end{proof}
	
	We now move on to our second order computations. To start, we collect $\nabla^s$-covariant derivatives of the tensors $U,V$ and $W$. Again, to make the expressions more compact, let us introduce tensors $X$, $Y$, and $Z$ defined in a local unitary frame by
	\[
	X^i_{pl} := T^k_{pj}T^r_{kl}\B{T_{ij}^r},
	\qquad Y^i_{pl} := T^k_{rj}T^r_{kp}\B{T_{ij}^l}, \qquad Z^i_{pl} := T^k_{jp}T^i_{kr}\B{T^l_{jr}}.
	\]
	Straightforward observations yield:
	\begin{lem} \label{lem:XYZIdentities} The contractions of $X\otimes \B{\eta}$, $Y \otimes \B{\eta}$ and $Z\otimes \B{\eta}$ are given by,
		\begin{enumerate}
			\item 	$X_{pi}^i\B{\eta_p} = A,\qquad X_{il}^i\B{\eta_l} = 2A$;\\
			\item  $Y_{pi}^i\B{\eta_p} = B,\qquad Y_{il}^i\B{\eta_l} = 0$;\\
			\item $Z_{pi}^{i}\B{\eta_p} = A,\qquad Z_{il}^i\B{\eta_l} = 0$.
		\end{enumerate}
	\end{lem}
	This notation now allows us to succinctly express the covariant derivatives of $U,V$ and $W$.

	\begin{lem}
		\label{lem:UVWDerivatives}
		If $\nabla^s$ satisfies the torsion-less Bianchi identity \eqref{eqn_Bianchi0} and $s \notin \{0,1\}$, then the following hold for all indices $1\leq i,k,p,q \leq n$:
		\begin{enumerate}
			\item $cU^p_{q,l} = -c(s-2)T^p_{rl}U^r_q+ T^p_{qr}\left((a-b)V_{l\B r} + bU^r_l + s^3W^{r\B l} - s^3\B{U^l_r}\right)$;
			\item $cU^k_{i,\B l} = aU^r_i\B{T^r_{kl}}- bU^k_r\B{T^i_{rl}}- s^3U^l_{r}\B{T^i_{rk}}- c(s-2)\B{Y_{lk}^i}$;
			\item $cV_{p\B i,l} = c\B{V_{i \B p,\B{l}}} = \left(a-c(s-2)\right)X_{pl}^i - bX_{lp}^i - bV_{p\B r}T^i_{rl} - s^3Y_{pl}^i + s^3Z_{pl}^i$;
			\item $ cW^{p\B k}_{, l} = c\B{W^{k\B p}_{,\B{l}}} = aW^{p\B r}T^r_{kl} - 2 bZ_{lk}^p - 2 s^3Z_{kl}^p - c(s-2)W^{r\B k}T^p_{rl}$.
	\end{enumerate}
	\end{lem}
	\begin{proof}

		For each item, we repeatedly apply \Cref{lem:torsionIdentities} and \Cref{cor:etaIdentities}.
		To compute item (1), we use \Cref{eqn:etaBarDeriv} to get
			\begin{align*}
				cU^p_{q,l} = cT^p_{qj,l}\B{\eta_j} + cT^p_{qj}\B{\eta_{j,\B l}} = -c(s-2)T^r_{qj}T^p_{rl}\B{\eta_j}+ T^p_{qj}\left((a-b)V_{l\B j}T + bU^j_l + s^3W^{j\B l} - s^3\B{U^l_j}\right),\\
			\end{align*}
		as required. The second item follows similarly. Indeed,
		\begin{align*}
			cU^k_{i,\B l} =& \,  cT^k_{ij,\B l}\B{\eta_j} + cT^k_{ij}\B{\eta_{j,l}}\\
			=& \,  \left(aT_{i j}^{r} \B{T_{k l}^{r}}+b\left(T_{i r}^{k} \B{T_{l r}^{j}}-T_{j r}^{k} \B{T_{l r}^{i}}\right)+ s^{3}\left(T_{i r}^{l} \B{T_{k r}^{j}}-T_{j r}^{l} \B{T_{k r}^{i}}\right)\right)\B{\eta_j} - c(s-2)T^k_{ij}\B{T^p_{rj}T^r_{pl}}\\
			=& \,   aU^r_i\B{T^r_{kl}} + 0 - bU^k_r\B{T^i_{rl}} + 0 - s^3U^l_{r}\B{T^i_{rk}}- c(s-2)T^k_{ij}\B{T^u_{rj}T^r_{ul}},
		\end{align*}
	which gives (2). Once again, to see (3), we have
	\begin{align*}
		cV_{p\B{i},l} &=  c\B{T^k_{ij,\B l}}T^k_{p j} + c\B{T^k_{ij}}T^k_{pj,l}\\
		&= \left(a\B{T_{i j}^{r}} T_{k l}^{r}+b\left(\B{T_{i r}^{k}} T_{l r}^{j}-\B{T_{j r}^{k}} T_{l r}^{i}\right)+ s^{3}\left(\B{T_{i r}^{l}} T_{k r}^{j}-\B{T_{j r}^{l}} T_{k r}^{i}\right)\right)T^k_{pj} - c(s-2)\B{T^k_{ij}}T_{pj}^rT_{rl}^k\\
		&= aX_{pl}^i - bX_{lp}^i - b\B{V_{r\B p}}T^i_{rl} - s^3Y_{pl}^i + s^3Z_{pl}^i - c(s-2)X_{pl}^i.
	\end{align*}
	Finally, to see (4), we have
	\begin{align*}
		cW^{p\B k}_{,l} &= c\B{T^k_{ij,\B l}}T^p_{ij} + c\B{T^k_{ij}}T^p_{ij,l}\\
		&= \left(a \, \B{T_{i j}^{r}} T_{k l}^{r}+b\left(\B{T_{i r}^{k}} T_{l r}^{j}-\B{T_{j r}^{k}} T_{l r}^{i}\right)+ s^{3}\left(\B{T_{i r}^{l}} T_{k r}^{j}-\B{T_{j r}^{l}}T_{k r}^{i}\right)\right)T^p_{ij} - c(s-2)\B{T^k_{ij}}T^r_{ij}T^p_{rl}\\
		&= aW^{p\B r}T^r_{kl} - bZ_{lk}^p - bZ_{lk}^p - s^3Z_{kl}^p - s^3Z_{kl}^p - c(s-2)W^{r\B k}T^p_{rl},
	\end{align*}
	as was to be shown.
	\end{proof}
	We now move on to computing certain contractions of $\big((\nabla^s)^2 T \big)\otimes \eta$:
	
	\begin{lem}\label{lem:etaSecondDeriv}
	If $\nabla^s$ satisfies the torsion-less Bianchi identity \eqref{eqn_Bianchi0} and $s \notin \{0,1\}$, then the following  hold:
	\begin{enumerate}
		\item $c^2 \, \eta_{j,\B l l}  \, \B{\eta_j} = \left((a-b)(-c(s-2) + a - 2b + 2s^3) - 2as^3\right)A+(2b - a)s^3B-\left(b^2+bs^3+s^6\right)C$ ; \\
		\item $c^2\eta_{j,\B j l}\B{\eta_l} = c^2\B{\eta_{j,\B j \B l}\eta_l} =2(a-b+s^3)(a-b-c(s-2))A+ c(b-s^3)\bgm{\partial|\eta|^2}{\B \eta}$;\\
		\item $\begin{aligned}c^2\B{\eta_{j,\B l \B j}\eta_l} =&\left((a-b)(a-3b + s^3 - c(s-2)) -2s^3(a+b+s^3)\right)A -\left((a-b)s^3 + b^2)\right)B \\&+ \left(b(a-b + s^3) + s^3(a+s^3)\right)C ;\end{aligned}$\\
		\item $c^2 \B{T^k_{ij,\B{i}\B j}\eta_k} = \left((a+s^3)(-a-s^3 + c(s-2)) - s^3(a-b+2s^3)\right)A + s^3(a-b+s^3)B + s^6C$.
	\end{enumerate} 
	\end{lem}
	\begin{proof}

		For the first item, using \Cref{lem:UVWDerivatives} and \Cref{lem:XYZIdentities}, we have
		\begin{align*}
			c^2\eta_{j,\B l l}\B{\eta_j} &=  c\left((a - b)V_{j\B l} + b  \B{U^j_l} + s^3 \, (W^{l\B j} - U^l_j)\right)_{,l}\B{\eta_j}\\
			&= c(a-b)V_{p\B i, i}\B{\eta_p} + cb\B{U^k_{l,\B{l}} \eta_k}+ cs^3W^{l\B k}_{,l}\B{\eta_k} - cs^3U^l_{j,l}\B{\eta_j}\\
			&=  (a-b)\left(aX_{pi}^i - bX_{ip}^i - b\B{V_{r\B p}}T^i_{ri} - s^3Y_{pi}^i + s^3Z_{pi}^i - c(s-2)X_{pi}^i\right)\B{\eta_p}\\
			&  \quad +b\left(a\B{U^r_l}T^r_{kl}- b\B{U^k_r}T^l_{rl}- s^3\B{U^l_{r}}T^l_{rk}- c(s-2)Y_{lk}^l\right)\B{\eta_k}\\
			& \quad   +s^3\left(aW^{l\B r}T^r_{kl} - 2 bZ_{lk}^l - 2 s^3Z_{kl}^l - c(s-2)W^{r\B k}T^l_{rl}\right)\B{\eta_k}\\
			&  \quad  -s^3\left(-c(s-2)T^l_{rl}U^r_q+ T^l_{qr}\left((a-b)V_{l\B r} + bU^r_l + s^3W^{r\B l} - s^3\B{U^l_r}\right)\right)\B{\eta_q}\\
			&=  a(a-b)A - 2b(a-b)A + b(a-b)C- s^3(a-b)B + (a-b)s^3A - c(a-b)(s-2)A\\
			& \quad  -abC + 0 - bs^3C + 0\\
			& \quad  -2as^3A  +0 - 2s^6A + 0\\
			& \quad   + 0 + (a-b)s^3A + bs^3B + 2s^6A - s^6C\\
			&=  \left((a-b)(-c(s-2) + a - 2b + 2s^3) - 2as^3\right)A + (2b-a)s^3B - (b^2+bs^3+s^6)C.
		\end{align*}
		For Item (2), we first note that by \Cref{lem:torsionIdentities}, 
		\[
			c^2\eta_{j,\B j} = c(a-b+s^3)|T|^2 + c(b-s^3)|\eta|^2.
		\]
		This is real, which proves the first equality. To see the second, we compute another derivative and apply \Cref{lem:normOfTDeriv} to yield,
		\begin{align*}
			c^2\eta_{j,\B jl}\B{\eta_l} &= \bgm{\partial c^2\eta_{j,\B j}}{\B{\eta}}\\
			&=  c(a-b+s^3)\bgm{\partial|T|^2}{\B \eta} + c(b-s^3)\bgm{\partial |\eta|^2}{\B \eta}\\
			&=2(a-b+s^3)(a-b-c(s-2))A + c(b-s^3)\bgm{\partial|\eta|^2}{\B \eta}.
		\end{align*}
		For Item (3), we proceed similarly to the proof of Item (1). Indeed, we have
		\begin{align*}
			c^2\B{\eta_{j,\B l\B j}\eta_l} &=   c\left((a - b)V_{l\B j} + b  U^j_l + s^3 \, (W^{j\B l} - \B{U^l_j})\right)_{,j}\B{\eta_l}\\
			&=  (a-b)cV_{p\B i,i}\B{\eta_p} + cbU^p_{q,p}\B{\eta_q} + s^3cW^{p\B k}_{,p}\B{\eta_k} - s^3c\B{U^k_{i,\B i}}\B{\eta_k}\\
			&=  (a-b)\left(aX_{pi}^i - bX_{ip}^i - bV_{p\B r}T^i_{ri} - s^3Y_{pi}^i + s^3Z_{pi}^i - c(s-2)X_{pi}^i\right)\B{\eta_p}\\
			&  \quad +b\left(-c(s-2)T^p_{rp}U^r_q+ T^p_{qr}\left((a-b)V_{p\B r} + bU^r_p + s^3W^{r\B p} - s^3\B{U^p_r}\right)\right)\B{\eta_q}\\
			&  \quad +s^3\left(aW^{p\B r}T^r_{kp} - 2 bZ_{pk}^p - 2 s^3Z_{kp}^p - c(s-2)W^{r\B k}T^p_{rp}\right)\B{\eta_k}\\
			& \quad   -s^3\left(a\B{U^r_i}T^r_{ki}- b\B{U^k_r}T^i_{ri}- s^3\B{U^i_{r}}T^i_{rk}- c(s-2)Y_{ik}^i\right)\B{\eta_k}\\
			&=  a(a-b)A - 2b(a-b)A + b(a-b)C - s^3(a-b)B + s^3(a-b)A - c(s-2)(a-b)A\\
			& \quad   + 0 - b(a-b)A - b^2B - 2bs^3A + bs^3C\\
			&  \quad-2as^3A + 0 -2s^6A + 0\\
			& \quad  +as^3C + 0 +s^6C + 0.\\
			&=   \left((a-b)(a-3b + s^3 - c(s-2)) -2s^3(a+b+s^3)\right)A -\left((a-b)s^3 + b^2\right)B \\
			&\quad+ \left(b(a-b+s^3)+s^3(a+s^3)\right)C.
		\end{align*}
		Finally, to prove Item (4), we first note that
		\[
		cT^k_{ij,\B{i}} = -(a+s^3)V_{j\B k} + b(W^{k\B j} - U^k_j) + s^3\B{U^j_k}.
		\]
		Thus,
		\begin{align*}
			c^2\B{T^k_{ij,\B{i}\B j}\eta_k} &= c\left(-(a+s^3)V_{k\B j} + b(W^{j\B k} - \B{U^k_j}) + s^3U^j_k)\right)_{,j}\B{\eta_k}\\
			&=-c(a+s^3)V_{p\B i,i}\B{\eta_p} + bcW^{p\B k}_{,p}\B{\eta_k} - bc\B{U^k_{i,\B i}}\B{\eta_k} + cs^3U^p_{q, p}\B{\eta_q}\\
			&=-(a+s^3)\left(aX_{pi}^i - bX_{ip}^i - bV_{p\B r}T^i_{ri} - s^3Y_{pi}^i + s^3Z_{pi}^i - c(s-2)X_{pi}^i\right)\B{\eta_p}\\
			&\quad+b\left(aW^{p\B r}T^r_{kp} - 2 bZ_{pk}^p - 2 s^3Z_{kp}^p - c(s-2)W^{r\B k}T^p_{rp}\right)\B{\eta_k}\\
			&\quad-b\left(a\B{U^r_i}T^r_{ki}- b\B{U^k_r}T^i_{ri}- s^3\B{U^i_{r}}T^i_{rk}- c(s-2)\B{Y_{ik}^i}\right)\B{\eta_k}\\
			&\quad+s^3\left(-c(s-2)T^p_{rp}U^r_q+ T^p_{qr}\left((a-b)V_{p\B r} + bU^r_p + s^3W^{r\B p} - s^3\B{U^p_r}\right)\right)\B{\eta_q}\\
			&= -a(a+s^3)A + 2b(a+s^3)A - b(a+s^3)C + s^3(a+s^3)B - s^3(a+s^3)A + c(s-2)(a+s^3)A\\
			&\quad-2abA + 0 - 2bs^3A + 0\\
			&\quad+abC + 0 + bs^3C + 0\\
			&\quad+0  - s^3(a-b)A - s^3bB - 2s^6A + s^6C\\
			&= \left((a+s^3)(-a-s^3 + c(s-2)) - s^3(a-b+2s^3)\right)A + s^3(a-b+s^3)B + s^6C, 
		\end{align*}
		as required.
		\end{proof}

	Since we now have formulas for the second covariant derivatives of $\eta$, we can use K\"ahler-like condition to derive new algebraic identities.
	\begin{lem}\label{lem:linearSystem}
		 If $\nabla^s$ is K\"ahler-like and $s\notin\{0,1\}$, then
		\begin{align}\label{eqn_linearsystem}
			\begin{pmatrix}
				0&c(s-2)&(b-a)&c\\
				x_1&x_2&x_3&-c(b-s^3)\\
				y_1&y_2&y_3&-c(b-s^3)\\
				z_1&z_2&z_3&0
			\end{pmatrix}\begin{pmatrix}
				A\\B\\C\\\bgm{\partial |\eta|^2}{\B\eta}
			\end{pmatrix} = 0,
		\end{align}
	where
	\begin{align*}
		x_1 :=&\,-\left(a-b+2 s^3\right)\left(a+b+c s+s^3\right);\\
		x_2 :=&\, 2bc(1-s) - (a-b)s^3 - b^2;\\
		x_3 :=&\, b(a-b+s^3)+s^3(a+s^3-2c(1-s));\\
		y_1 :=&\, -\left(a-b+2s^3\right) (a-2 c (s-1));\\
		y_2 :=&\, s^3(2b-a-c s) + c^2s(s-2);\\
		y_3 :=&\, acs - b^2 - bs^3 - s^6;\\
		z_1 :=&\, -(a+s^3)(a+s^3 + cs) - s^3(a-b+2s^3);\\
		z_2 :=&\, s^3(a-b+s^3);\\
		z_3 :=&\, s^6.\\
	\end{align*}
	In particular, the determinant of this system is
	\[
	\det \begin{pmatrix}
		0&c(s-2)&(b-a)&c\\
		x_1&x_2&x_3&-c(b-s^3)\\
		y_1&y_2&y_3&-c(b-s^3)\\
		z_1&z_2&z_3&0
	\end{pmatrix} = 64 \,  s^8 (s-2)^3 (s-1)^3 (2 s-1)^3 (3s-2)^2 (5s-4). 
	\]
	\end{lem}
	\begin{proof}
		The first equation is precisely \Cref{lem:normOfTDeriv}, \Cref{eqn_deleta^2}. For the second row, we have by the type condition and \Cref{cor:secondDerivComm} that 
		\[
			0 = c^2\bgm{R^s(e_j,\B{\eta^\sharp})\B{\eta^\sharp}}{\B{e_j}} = c^2\left(\B{\eta_{j,\B l \B j}} - \B{\eta_{j, \B j \B l}} + 2(1-s)T_{j l}^k\B{\eta_{j,\B k}}\right)\B{\eta_l}.
		\]
		Computing the first order term, we have
		\begin{align*}
			2c^2(1-s)T^k_{jl}\B{\eta_{j,\B k}}\B{\eta_l} &= 2c^2(1-s)T^l_{jk}\B{\eta_{j,\B l}}\B{\eta_k}\\
			&=   2c(1-s)T^l_{jk}\left((a - b)V_{l\B j} + b  U^j_l + s^3 \, (W^{j\B l} - \B{U^l_j})\right)\B{\eta_k}\\
			&=   2(a-b)c(1-s)A + 2bc(1-s)B + 4c(1-s)s^3A - 2cs^3(1-s)C\\
			&=  2c(1-s)(a-b + 2s^3)A + 2bc(1-s)B - 2cs^3(1-s)C.
		\end{align*}
	Thus, by \Cref{lem:etaSecondDeriv},
	\begin{align*}
		0 &= c^2\B{\eta_{j,\B{l j}}\eta_l} - c^2\B{\eta_{j,\B {j l}} \eta_l} + 2c^2(1-s)T^k_{jl}\B{\eta_{j,\B k}}\B{\eta_l}\\
		&= \left((a-b)(a-3b + s^3 - c(s-2)) -2s^3(a+b+s^3)\right)A -\left((a-b)s^3 + b^2)\right)B \\
		& \quad + \left(b(a-b + s^3) + s^3(a+s^3)\right)C\\
		&  \quad-2(a-b+s^3)(a-b-c(s-2))A- c(b-s^3)\bgm{\partial|\eta|^2}{\B \eta}\\
		& \quad  + 2c(1-s)(a-b + 2s^3)A + 2bc(1-s)B - 2cs^3(1-s)C\\
		&= -\left(a-b+2 s^3\right)\left(a+b+c s+s^3\right)A + \left(2bc(1-s) - (a-b)s^3 - b^2\right)B\\
		 &\quad   +\left(b(a-b+s^3)+s^3(a+s^3-2c(1-s))\right)C- c(b-s^3)\bgm{\partial|\eta|^2}{\B \eta}\\
		&= x_1A+x_2B+x_3C - c(b-s^3)\bgm{\partial|\eta|^2}{\B \eta}.
	\end{align*}
	For the third row, we use the torsion-less Bianchi identity \eqref{eqn_Bianchi0} for $R^s$ and \Cref{cor:secondDerivComm} to get
	\begin{align*}
		0 &= \, c^2\bgm{R^s(e_{l},\B{e_l})\eta^\sharp}{\B{\eta^\sharp}}  - c^2\bgm{R^s(\B{\eta^\sharp},\B{e_l})\eta^\sharp}{e_l}\\
		&=  \, c^2\left(\eta_{j,\B l l} - \eta_{j, l \B l} - s\B{\eta_k}\eta_{j,k} + s\eta_k\eta_{j,\B k}\right)\B{\eta_j}\\
		& \quad -c^2\left(\eta_{l,\B l j} - \eta_{l,j\B l} - s\B{T^j_{lk}}\eta_{l,k} + sT^l_{jk}\eta_{l,\B k}\right)\B{\eta_j}\\
		&=  \, c^2\left(\eta_{j,\B l l} - \eta_{l,\B l j}\right)\B{\eta_j}+ c^2s\left(\eta_k\eta_{j,\B k} - \B{\eta_k}\eta_{j,k} - T^l_{jk}\eta_{l,\B k}\right)\B{\eta_j},
	\end{align*}
	where in the last line, we used the fact that $\eta_{j,l} = \eta_{l,j}$ by \Cref{cor:etaIdentities}. Applying \Cref{cor:etaIdentities} again, the first order term is given by
	\begin{align*}
		c^2s\left(\eta_k\eta_{j,\B k} - \B{\eta_k}\eta_{j,k} - T^l_{jk}\eta_{l,\B k}\right)\B{\eta_j} &= cs\eta_l\left((a - b)V_{j\B l} + b  \B{U^j_l} + s^3 \, (W^{l\B j} - U^l_j)\right)\B{\eta_j}\\
		&  \quad +c^2s(s-2)\B{\eta_k}T^r_{lj}T^l_{rk}\B{\eta_j}\\
		& \quad  -csT^j_{kl}\left((a - b)V_{j\B l} + b  \B{U^j_l} + s^3 \, (W^{l\B j} - U^l_j)\right)\B{\eta_k}\\
		&= cs(a-b)C + 0 + 0 - 0\\
		&   \quad +c^2s(s-2)B\\
		&  \quad  + cs(a-b)A + bcsC + 2cs^4A - cs^4B\\
		&= cs(a-b + 2s^3)A + cs(c(s-2)-s^3)B+ acsC.
	\end{align*}
	Hence, substituting this back into the above equation and using \Cref{lem:etaSecondDeriv}, we have
	\begin{align*}
		0 &= \left((a-b)(-c(s-2) + a - 2b + 2s^3) - 2as^3\right)A+(2b - a)s^3B-\left(b^2+bs^3+s^6\right)C\\
		& \quad -2(a-b+s^3)(a-b-c(s-2))A- c(b-s^3)\bgm{\partial|\eta|^2}{\B \eta}\\
		&  \quad+cs(a-b + 2s^3)A + cs(c(s-2)-s^3)B+ acsC\\
		&= -\left(a-b+2s^3\right) (a-2 c (s-1))A + \left(s^3(2b-a-c s) + c^2s(s-2)\right)B \\
		& \quad + \left(acs - b^2 - bs^3 - s^6\right)C - c(b-s^3)\bgm{\partial |\eta|^2}{\B{\eta}}\\
		&=  y_1A + y_2B + y_3C - c(b-s^3)\bgm{\partial |\eta|^2}{\B{\eta}}
	\end{align*}
	Finally, to establish the last equation, we again use the type condition: for all $1 \leq i,j\leq n$ it holds that $R^s(e_i,e_j) = 0$. In particular, the same is true when applying $R^s$ to the torsion tensor $T$. Thus, by \Cref{cor:secondDerivComm} and \Cref{lem:etaSecondDeriv},
	\begin{align*}
		0 &= c^2\B{\eta\left((R^s(\B{e_j},\B{e_i})T)(e_i,e_j)\right)}\\ 
		&= \frac{c^2}{2}\B{\left(T_{ij,\B i \B j}^k - T_{ij,\B j \B i}^k + 2(1-s)\B{T_{ji}^l}T^k_{ij,\B l}\right)}\B{\eta_k} \\
		&= c^2\B{T^k_{ij,\B{ij}}\eta_k} + c^2(s-1)T^l_{ij}\B{T^k_{ij,\B l}\eta_k}\\
		&= \left((a+s^3)(-a-s^3 + c(s-2)) - s^3(a-b+2s^3)\right)A + s^3(a-b+s^3)B + s^6C\\
		&  \quad + c(s-1)T^l_{ij}\left(a \, \B{T_{i j}^{r}} T_{k l}^{r}+b\left(\B{T_{i r}^{k}} T_{l r}^{j}-\B{T_{j r}^{k}} T_{l r}^{i}\right)+ s^{3}\left(\B{T_{i r}^{l}} T_{k r}^{j}-\B{T_{j r}^{l}}T_{k r}^{i}\right)\right)\B{\eta_k}\\
		&=  \left((a+s^3)(a-s^3 + c(s-2)) - s^3(a-b+2s^3)\right)A + s^3(a-b+s^3)B + s^6C\\
		& \quad  -2ac(s-1)A + 0 + 0 -cs^3(s-1)A - cs^3(s-1)A\\
		&=   -\left((a+s^3)(a+s^3 + cs) + s^3(a-b+2s^3)\right)A + s^3(a-b+s^3)B + s^6C\\
		&=    z_1A + z_2B + z_3C,
	\end{align*}
	as required.
	\end{proof}	

	\begin{proof}[Proof of Theorem \ref{thm_main_body}] 
First notice that thanks to \Cref{cor:etaIdentities}, it suffices to show that $\eta \equiv 0$. Moreover, to prove that $\eta$ vanishes at a point $p\in M$, it is in turn enough to show $C=0$ at $p$. Indeed, since $C = |U|^2$, it follows that $U=0$ and thus $|\eta|^2 = \operatorname{tr}U = 0$. 

	We divide the proof into cases according to the different values of $s$. First recall that the case $s=1/2$ was established in \cite[~Lemma 3]{ZhaoZhengKahlerLike} (alternatively, for $s=1/2$ the first equation in \eqref{eqn_linearsystem} becomes $C=0$). Regarding $s=1$, by  \cite[~Proposition 1]{ZhaoZhengKahlerLike} we have $\eta = 0$ and thus we are done as well. 

	Next, we assume that $s \notin \{0, \frac{1}{2},\frac{2}{3},\frac{4}{5},1, 2\}$. For these values of $s$, the determinant of the system \eqref{eqn_linearsystem} is non-zero, thus $C=0$ everywhere and the result follows.

	Finally, we deal with the cases $s=2/3, 4/5$ assuming that $|\eta|^2$ attains its maximum at $p \in M$. Then at $p$, the system \eqref{eqn_linearsystem} becomes overdetermined, with $4$ equations for $3$ unknowns $A, B, C$. For $s=4/5$ the system has rank $3$, from which $C = 0$ at $p$. As above, this implies that $|\eta|^2 = 0$ at $p$, and since this was the maximum, we deduce $\eta \equiv 0$. For $s=2/3$, the rank is $2$, but still by looking at the first and third equations, they respectively become (up to scaling) 
	\[
		B + C = 0, \qquad B-C = 0,
	\]
	hence $C=0$ at $p$ and we may argue as in the previous case. The theorem follows.
	\end{proof}
	\bibliography{GauduchonFlat.bib}
	\bibliographystyle{amsalpha}
\end{document}